\documentclass[11pt]{amsart}
\usepackage{etex}
\usepackage{amsfonts,amsmath,amssymb,amsthm,amscd,amsxtra}
\usepackage{enumerate,verbatim}

\usepackage[all,2cell,ps]{xy}
\usepackage{nicefrac}
\usepackage{array}
\usepackage{enumerate,verbatim}
\usepackage{amsfonts,amsmath,amssymb,amsthm,amscd,amsxtra}
\usepackage[notcite,notref]{}
\usepackage{amssymb,amstext,amsmath,amscd,amsthm,amsfonts,enumerate,graphicx,latexsym}
\usepackage[usenames]{color}

\usepackage{amsfonts,amsmath,amssymb,amsthm,amscd,amsxtra}
\usepackage{todonotes}

\usepackage[all,2cell,ps]{xy}
\usepackage[notcite,notref]{}
\usepackage[pagebackref]{hyperref}
\usepackage{mathptmx}

\theoremstyle{plain} 
\newtheorem{thm}{Theorem}[section]
\newtheorem{chunk}[thm]{\hspace*{-1.065ex}\bf}
\newtheorem{prop}[thm]{Proposition}
\newtheorem{cor}[thm]{Corollary}

\newtheorem{lem}[thm]{Lemma}

\theoremstyle{definition}
\newtheorem{dfn}[thm]{Definition}

\newtheorem{conj}[thm]{Conjecture}

\newtheorem{rmk}[thm]{Remark}

\numberwithin{equation}{section}
\newcommand{\lra}{\longrightarrow}

\newcommand{\fm}{\mathfrak{m}}
\newcommand{\fn}{\mathfrak{n}}
\newcommand{\fp}{\mathfrak{p}}

\newcommand{\up}[1]{{{}^{#1}\!}}

\DeclareMathOperator{\Ann}{Ann}

\DeclareMathOperator{\depth}{depth}

\DeclareMathOperator{\Ext}{Ext}

\DeclareMathOperator{\hh}{H}
\DeclareMathOperator{\Hom}{Hom}

\DeclareMathOperator{\id}{id}
\DeclareMathOperator{\im}{im}

\DeclareMathOperator{\pd}{pd}

\DeclareMathOperator{\Supp}{Supp}
\DeclareMathOperator{\Spec}{Spec}

\DeclareMathOperator{\Tor}{Tor}

\DeclareMathOperator{\Jac}{Jac}
\DeclareMathOperator{\grade}{grade}
\DeclareMathOperator{\PP}{\mathcal{P}}

\newcommand{\tp}[2]{{\boldsymbol\top}_{\hskip-2pt #1}{#2}}

 \renewcommand{\le}{\leqslant} 
\renewcommand{\geq}{\geqslant} \renewcommand{\leq}{\leqslant}

\def\urltilda{\kern -.15em\lower .7ex\hbox{\~{}}\kern .04em}
\def\urldot{\kern -.10em.\kern -.10em}\def\urlhttp{http\kern -.10em\lower -.1ex
\hbox{:}\kern -.12em\lower 0ex\hbox{/}\kern -.18em\lower 0ex\hbox{/}}

\begin{document}

\author[O. Celikbas]{Olgur Celikbas}
\address{Olgur Celikbas\\
School of Mathematical and Data Sciences\\
West Virginia University\\
Morgantown, WV 26506-6310, U.S.A}
\email{olgur.celikbas@math.wvu.edu}

\author[Y. Yongwei]{Yongwei Yao}
\address{Yongwei Yao\\
Department of Mathematics and Statistics, 
Georgia State University\\
Atlanta, GA 30303, U.S.A.}
\email{yyao@gsu.edu}


\title[On the vanishing of (co)homology]{On the vanishing of (co)homology for modules \\admitting certain filtrations}

\begin{abstract} We study the vanishing of (co)homology along ring homomorphisms for modules that admit certain filtrations, and  generalize a theorem of O.~Celikbas-Takahashi. Our work produces new classes of rigid and test modules, in particular over local rings of prime characteristic. It also gives applications in the study of torsion in tensor products of modules, for example, concerning the Huneke-Wiegand conjecture. 
\end{abstract}

\date{April 15, 2023}

\maketitle{}

\setcounter{tocdepth}{1}

\section{Introduction}

Throughout all rings are assumed to be commutative. By a local ring we mean a Noetherian ring with a unique maximal ideal. The research in this paper has been initiated by the following result: 

\begin{thm}[{O.~Celikbas-Takahashi \cite{CT21}}] \label{CT}
Let $R$ be a local ring and let $M$ be a finitely generated $R$-module such that $\depth_R(M)\geq 1$. If $\Tor_n^R(\fm^tM,N) = 0$ for some finitely generated $R$-module $N$, where $t\geq 0$ and $n\geq 0$, then
$\Tor_n^R(M, N) = 0$. 
\end{thm}

The main purpose of this paper is to extend Theorem~\ref{CT} by considering its conclusion along ring homomorphisms. Our main result, which extends Theorem~\ref{CT}, is concerned with the vanishing of (co)homology for modules (not necessarily finitely generated) that admit a certain filtration. More precisely, we prove:

\begin{thm} \label{mainthm} Let $f: R \to S$ be a ring homomorphism and let $M\neq 0$ be a finitely generated $S$-module. Assume the following hold:
\begin{enumerate}[\; \rm(i)]
\item $(R, \fm)$ is local.
\item $S$ is Noetherian and $\fm S \subseteq \Jac(S)$.
\item There is an element of $\fm$ which is a non zero-divisor on $M$. 
\item There is a filtration of $M$ of the form $M_t \le M_{t-1} \le \cdots \le M_1 \le M_0 = M$,
where each $M_i$ is an $S$-module, and $\fm M_{i-1} \le M_i \le \Jac(S)M_{i-1}$ for all $i = 1, \dotsc, t$. 
\end{enumerate} 
If $\Tor_n^R(M_i, N)= 0$, where $N$ is a finitely generated $R$-module, $n\geq 0$, and $1\leq i \leq t$, then it follows
\[
\Tor_n^R(M_{i-1}, N) = \Tor_n^R(M_{i-2}, N) =\cdots = \Tor_n^R(M, N)=0. 
\] 
\end{thm}

Theorem~\ref{mainthm} is subsumed by Theorem~\ref{Mainthm2}, which is proved in section 2; see Theorem~\ref{Mainthm3} for an $\Ext$ verison of the theorem. Let us note here that it is not difficult to reprove Theorem~\ref{CT} by using Theorem~\ref{mainthm}: if $R=S$ and $f$ is the identity map, then our assumption on the filtration of $M$ forces $M_i=\fm^iM$ for all $i\geq 0$, and hence Theorem~\ref{CT} follows.

If $R$ is a ring with prime characteristic $p$, $M$ is an $R$-module, and $F$ is the Frobenius map, then each iteration $F^e$ defines a new $R$-module structure on $M$, denoted by $\up{e}{M}$, where $r \cdot x=r^{p^e}x$ for $r\in R$ and $x\in M$. 

Assume $(R, \fm)$ is local ring of prime characteristic $p$. Then hypothesis (iv) in Theorem~\ref{mainthm} implies that  $\fm^{[p^e]} M_{i-1} \le M_i \le \fm M_{i-1}$ for all $i = 1, \dotsc, t$. One way to obtain such a filtration is to assume there are ideals $I_1, \ldots, I_t$ of $R$ such that $\fm^{[p^e]} \subseteq I_j \subseteq \fm$ for all $j=1, \ldots, t$, and set $M_i=(I_1\cdots I_i)M$ for all $i=1, \ldots, t$, which would imply $M_i=I_iM_{i-1}$. Hence we can apply Theorem~\ref{mainthm} for the special case where $f=F^e$ and conclude:

\begin{cor} \label{cor1intro} Let $(R, \fm)$ be a local ring of prime charactistic $p$, and let $M$ and $N$ be nonzero finitely generated $R$-modules such that $\depth_R(M)\geq 1$. Let $I_1, \ldots, I_t$ be ideals of $R$ such that $\fm^{[p^e]} \subseteq I_j \subseteq \fm$ for all $j=1, \ldots, t$, where $e\geq 0$ and $t\geq 1$.

If $\Tor_n(\up{e}{\big((I_1\cdots I_i)M\big)}, N)=0$ for some $n\geq 1$ and for some $i$, where $1\leq i\leq t$, then it follows: 
\[
\Tor_n^R(\up{e}{X_i}, N) = \Tor_n^R(\up{e}{\big((I_1\cdots I_{i-1})M\big)}, N) =\cdots = \Tor_n^R(\up{e}{M}, N)=0.
\] 
\end{cor}

If $M$ is a finitely generated $R$-module via $f$, then each filtration module $M_i$, with $1\leq i \leq t$, in Theorem~\ref{mainthm} is a test module over $R$; see Definition~\ref{def} and Corollary~\ref{cor3}. We make use of this fact and obtain new classes of test modules, for example, over rings of prime characteristic $p$. One such result is the following proposition, which is proved in Corollary~\ref{cor2}(ii).

\begin{prop} \label{cor2intro} Let $(R, \fm)$ be a local F-finite complete intersection ring of positive depth and prime characteristic $p$, $e\geq 0$, and let $t\geq 1$. If $\Tor_n^R(\up{e}{(\fm^t)},N)=0$ or $\Tor_n^R\Big(\up{e}{ \Big(\big(\fm^{[p^e]}\big)^t\Big)}, N\Big)=0$ for some finitely generated $R$-module $N$ and $n\geq 1$, then $\pd_R(N)<\infty$. 
\end{prop}


\section{A proof of Theorem~\ref{mainthm}}

The following proposition is key to proving Theorem~\ref{mainthm}.

\begin{prop} \label{Mainthm1} Let $f: R \to S$ be a ring homomorphism such that $(R, \fm)$ is local and $S$ is Noetherian. Let $M\neq 0$ be a 
finitely generated $S$-module and assume the following conditions hold:
\begin{enumerate}[\; \rm(i)]
\item $\fm S \subseteq \Jac(S)$.
\item There is a filtration of $M$ of the form $M_t \le M_{t-1} \le \dotsb \le M_1 \le M_0 = M$,
where each $M_i$ is an $S$-module and $\fm M_{i-1} \le M_i \le \Jac(S)M_{i-1}$ for all $i = 1, \dotsc, t$. 
\end{enumerate} 
Let $F_{\bullet}$ be a minimal complex of finitely generated free $R$-modules:
\[
F_{\bullet}: \qquad
\dotsb \lra F_{n+1} \overset{f_{n+1}}{\lra} F_n \overset{f_n}{\lra} F_{n-1} 
\lra \dotsb 
\]
\begin{enumerate}[\; \rm(1)]
\item Assume $H_n(M_{i}\otimes_R F_{\bullet}) =0$ for some integer $n$ and some $i$ with $1\leq i \leq t$. Then $1_{M_{i-1}} \otimes f_{n+1}=0$. Moreover, if $\grade_R(\fm , M)\geq 1$, then $H_n(M_{i-1}\otimes_R F_{\bullet}) = 0$.
\item If $H_n(M_{i}\otimes_R F_{\bullet})=0$ for some integer $n$ and $M_{i-1}$ is faithful as an $R$-module for some $i$ with $1\leq i \leq t$, then $f_{n+1}$ is the zero map.
\item If  $H_n(M_{i}\otimes_R F_{\bullet}) =  H_{n+1}(M_{i}\otimes_R F_{\bullet}) = 0$ for some integers $n$ and $i$ with $1\leq i \leq t$, then $F_{n+1}=0$.
\end{enumerate}
\end{prop}

\begin{proof} It suffices to prove the case when $i = 1$ (note each $M_i$ has positive grade if $M$ has positive grade). Tensoring $F_{\bullet}$ with $M$ over $R$, we obtain the complex 
\[
M \otimes_R F_{\bullet}: \qquad
\dotsb \lra C_{n+1} \overset{g_{n+1}}{\lra} C_n \overset{g_n}{\lra} C_{n-1} 
\lra \dotsb 
\]
where $C_i = M \otimes_RF_i$ and $g_i = 1_{M} \otimes f_i$.
Similarly, tensoring $F_{\bullet}$ with $M_1$ over $R$, we obtain the complex
\[
M_1 \otimes_R F_{\bullet}: \qquad
\dotsb \lra D_{n+1} \overset{h_{n+1}}{\lra} D_n \overset{h_n}{\lra} D_{n-1} 
\lra \dotsb. 
\]
where $D_i = M_1 \otimes_R F_i$ and $h_i = 1_{M_1} \otimes f_i$.

As $F_i$ is free over $R$, we may assume $D_i \le
C_i$ for all $i$ and $D_{\bullet}$ is a subcomplex of $C_{\bullet}$. Since $F_{\bullet}$ is minimal over $(R, \fm)$ and $\fm M \le M_{1}$,
we see that $\im(g_{n+1}) \le \fm C_n \le D_n$. 

(1) Assume $H_n(M_1 \otimes_R F_{\bullet}) = 0$, that is, $\ker(h_n) = 
\im(h_{n+1})$. It follows that $D_{n+1} \le \Jac(S) C_{n+1}$ because $M_1 \le \Jac(S)M$.
Therefore $\im(h_{n+1})\le \Jac(S)\im(g_{n+1})$. Putting it all together, we see
\[
\im(g_{n+1}) \le \ker(g_n) \cap \fm C_n \le \ker(g_n) \cap D_n = \ker(h_n) =
\im(h_{n+1}) \le \Jac(S)\im(g_{n+1}) \le \im(g_{n+1}), \tag{$*$}
\]
which shows that $\Jac(S)\im(g_{n+1}) = \im(g_{n+1})$. As $C_i$ is a finitely generated $S$-module for each $i$, we conclude that $\im(g_{n+1})
= 0$ by Nakayama's lemma. This proves that the function $1_{M} \otimes f_{n+1}$ is zero. Moreover, we see from $(*)$ that $\ker(g_n) \cap \fm C_n=0$, which implies that $\fm \ker(g_{n}) \le \ker(g_n) \cap \fm C_n=0$.
If $\grade_R(\fm , M)\geq 1$ (so that $\grade_R(\fm, C_n)\geq 1$), we must have 
$\ker(g_{n}) = 0$, and hence $H_n(M \otimes_R F_{\bullet}) =
\ker(g_n)/\im(g_{n+1}) = 0$. 

(2) Assume $H_n(M_{i}\otimes_R F_{\bullet})=0$ for some integer $n$ and $M$ is faithful as an $R$-module. Recall $1_{M} \otimes f_{n+1}=0$ by part (1). This implies that the ideal of $R$ generated by the entries of the matrix representing $f_{n+1}$ is contained in the annihilator of $M$. Therefore $f_{n+1}$ is the zero map.

(3) Assume $H_n(M_{1}\otimes_R F_{\bullet}) =  H_{n+1}(M_{1}\otimes_R F_{\bullet}) = 0$ for some integer $n$. Then it follows from part (1) that $$F_{n+1}\otimes_R M=\ker(1_M\otimes f_{n+1})=\im(1_M \otimes f_{n+2})=0.$$
As $M\neq 0$, this implies that $F_{n+1}=0$.
\end{proof}

\begin{rmk} The general idea used to prove the first two parts of Proposition~\ref{Mainthm1} stems from the work of Levin and Vascencelos \cite[Lemma, page 316]{LV} (see also \cite[the proof of 1.2 and 2.2]{CT21}). Moreover the argument in the proof of part (3) of the proposition is essentially from the proof of \cite[2.5(1)]{DeyTos}.
\end{rmk}

Theorem~\ref{mainthm}, which is advertised in the introduction, is subsumed by the next result.

\begin{thm} \label{Mainthm2} Let $f: R \to S$ be a ring homomorphism such that $(R, \fm)$ is local and $S$ is Noetherian. Let $M\neq 0$ be a
finitely generated $S$-module and assume the following conditions hold:
\begin{enumerate}[\; \rm(i)]
\item $\fm S \subseteq \Jac(S)$.
\item There is a filtration of $M$ of the form $M_t \le M_{t-1} \le \dotsb \le M_1 \le M_0 = M$,
where each $M_i$ is an $S$-module and $\fm M_{i-1} \le M_i \le \Jac(S)M_{i-1}$ for all $i = 1, \dotsc, t$. 
\end{enumerate}  
Then, for each  $n\geq 0$, and $1\leq i \leq t$, and for each finitely generated $R$-module $N$, we have the following:
\begin{enumerate}[\; \rm(a)]
\item If $\grade_R(\fm , M)\geq 1$ and $\Tor_n^R(M_i, N)= 0$, then $\Tor_n^R(M_{i-1}, N) = 0$.
\item If $\Tor_n^R(M_{i}, N)= 0$ and $M_{i-1}$ is a faithful $R$-module, then $\pd_R(N)\leq n$.
\item Assume $\Tor_n^R(M_i, N)= \Tor_{n+1}^R(M_i, N)= 0$. Then $\pd_R(N)\leq n$. Moreover, if $n=0$, then $N=0$.
\end{enumerate}
\end{thm}

\begin{proof} The claims follow from Proposition~\ref{Mainthm1} by letting $F_{\bullet}$ be a minimal free resolution of $N$ over $R$. 
\end{proof}

One can obtain a version of Theorem~\ref{Mainthm2} in terms of the Ext modules. 

\begin{thm} \label{Mainthm3} Assume the same setup as in Theorem~\ref{Mainthm2}. Then, for each  $n\geq 0$, and $1\leq i \leq t$, and for each finitely generated $R$-module $N$, we have the following:
\begin{enumerate}[\; \rm(a)]
\item If $\grade_R(\fm , M)\geq 1$ and $\Ext^n_R(N, M_i)= 0$, then $\Ext^n_R(N, M_{i-1}) = 0$.
\item If $\Ext^n_R(N, M_{i})= 0$ and $M_{i-1}$ is a faithful $R$-module, then $\pd_R(N)\leq n-1$.
\item If $\Ext^n_R(N,M_i)= \Ext^{n+1}_R(N,M_i)= 0$, then $\pd_R(N)\leq n-1$.
\end{enumerate}
\end{thm}

\begin{proof} Let $G_{\bullet}$ be the minimal free resolution of $N$ over $R$. Then $\Hom_R(G_{\bullet}, M) \cong \Hom_R(G_{\bullet}, R)\otimes_R M$ as complexes. As $\Ext^n_R(N, M_i)= \hh_{-n}(\Hom_R(G_{\bullet}, R)\otimes_R M)$, we can apply Proposition~\ref{Mainthm1} with the complex $\Hom_R(G_{\bullet}, R)$ and establish the claims.
\end{proof}

In the next section we establish some corollaries of Theorem~\ref{Mainthm2} and extend several results from the literature concerning rigid and test modules. 


\section{On rigid and test modules}

\begin{dfn} \label{def} Let $R$ be a ring and let $M\neq 0$ be a finitely generated $R$-module. 
\begin{enumerate}[\rm(i)]
\item $M$ is called a \emph{test module} (or a $\pd$-test module) provided that the following condition holds: if $N$ is a finitely generated $R$-module and $\Tor_i^R(M,N)=0$ for all $i\gg 0$, then $\pd_R(N)<\infty$; see \cite[1.1]{Test1}.
\item $M$ is called \emph{Tor-rigid} provided that the following condition holds: if $N$ is a finitely generated  $R$-module and $\Tor_1^R(M,N)=0$, then $\Tor_2^R(M,N)=0$; see \cite{Au}.
\item $M$ is called \emph{rigid-test} provided that $M$ is both test and Tor-rigid \cite[2.3]{HD2}.
\item $M$ is called \emph{strongly-rigid} provided that the following condition holds: if $N$ is a finitely generated $R$-module and $\Tor_n^R(M,N)=0$ for some $n\geq 1$, then $\pd_R(N)<\infty$; see \cite[2.1]{DLM}.
\end{enumerate}
\end{dfn}

There are various classes of test and rigid modules in the literature. For example, finitely generated modules are Tor-rigid over regular local rings, and finitely generated modules of infinite projective dimension are test over hypersurface rings; see, \cite[1.9]{HW2} for the details. Note that, by definition, each strongly-rigid module is a test module. Similarly, each rigid-test module is strongly-rigid, but do not know if the converse is true. More precisely, we do not know if each strongly-rigid module is Tor-rigid, in general. On the other hand, in light of the Auslander-Buchsbaum formula, if $R$ is a local ring of depth at most one and $M$ is a finitely generated $R$-module, then $M$ is strongly-rigid if and only if $M$ is rigid-test.

Our aim is to make use of Theorem~\ref{mainthm} and obtain new classes of rigid and test modules, and extend some known results in this direction. Let us note that Theorem~\ref{mainthm} allows one to generalize many results (for example results from \cite{HD2}) along ring homomorphisms; here we obtain only a few such results to demonstrate some useful applications of Theorem~\ref{mainthm}. 

We first turn our attention to the following beautiful result of Levin and Vascencelos:

\begin{chunk} [{Levin-Vascencelos, \cite[1.1 and Lemma, page 316]{LV}}] \label{LV} Let $(R, \fm)$ be  a local ring and $M$ and $N$ are finitely generated $R$-modules such that $\fm M\neq 0$. If $\Tor_n^R(\fm M, N)=\Tor_{n+1}^R(\fm M, N)=0$ for some $n\geq 0$, then $\pd_R(N)\leq n$. Therefore, $\fm M\neq 0$ is a test module; see also \cite[2.9]{CIST}. 
\end{chunk}

As mentioned in the introduction, for the special case where $R=S$ and $f$ is the identity map, our assumption on the filtration of $M$ in Theorem~\ref{mainthm} forces $M_i=\fm^iM$ for each $i\geq 0$. Therefore the next corollary yields an extension of the result of Levin and Vascencelos recorded in~\ref{LV}. The corollary also shows that  certain characterizations of local rings in terms of test modules carry over along ring homomorphisms:


\begin{cor} \label{cor3} Let $f: R \to S$ be a ring homomorphism such that $(R, \fm)$ is local, $S$ is Noetherian, and $\fm S \subseteq \Jac(S)$. Let $M$ be a finitely generated $S$-module. Assume there is a filtration of $M$ of the form
\[
M_t \le M_{t-1} \le \dotsb \le M_1 \le M_0 = M,
\] 
where each $M_i$ is a nonzero $S$-module and $\fm M_{i-1} \le M_i \le \Jac(S)M_{i-1}$ for all $i = 1, \dotsc, t$. 
Let $H=M_j$, for some $j$, where $1\le j \le t$. 
\begin{enumerate}[\rm(1)]
\item The following statements are equivalent:
\begin{enumerate}[\rm(i)]
\item $R$ is regular.
\item $\pd_R(H)<\infty$.
\item $\Tor_n^R(H, T)= \Tor_{n+1}^R(H, T)= 0$ for some $n\geq 0$ and some test module $T$ over $R$.
\item $\id_R(H)<\infty$.
\item $\Ext^n_R(T, H)=\Ext_R^{n+1}(T, H)= 0$ for some $n\geq 0$ and some test module $T$ over $R$.
\end{enumerate}
\item If $H$ is a finitely generated $R$-module (e.g., $M$ is finitely generated $R$-module via $f$), then $H$ is a test module over $R$. 
\end{enumerate}
\end{cor}

\begin{proof} (1) It is clear that part (i) implies part (ii), and part (ii) implies part (iii). If part (iii) holds, then Theorem~\ref{Mainthm2}(c) implies that $\pd_R(T)<\infty$ so that $R$ is regular since $\Tor_i^R(T,R/\fm)=0$ for all $i\gg 0$. This shows the equivalence of parts (i), (ii) and (iii). As part (i) implies part (iv), and part (iv) implies part (v), it suffices to prove that part (v) implies part (i), which follows from  Theorem~\ref{Mainthm3}(c). 

(2) This follows from Definition~\ref{def}(i) and Theorem~\ref{Mainthm2}(c).
\end{proof}

The following observation is used in Corollaries~\ref{rmk1} and~\ref{cor5}, and Lemma~\ref{LT}.

\begin{chunk} \label{olsun} Let $f: (R, \fm) \to (S, \fn)$ be a local ring homomorphism of local rings and let $M$ be a finitely generated $S$-module that is finitely generated $R$-module via $f$. Assume $M$ admits a filtration of the form
\[
M_t \le M_{t-1} \le \dotsb \le M_1 \le M_0 = M,
\] 
where each $M_i$ is a nonzero $S$-module and $\fm M_{i-1} \le M_i$ for all $i = 1, \dotsc, t$. Then, for each integer $i$ with $1 \leq i \leq t$, $M_{i-1}/M_i$ is annihilated by $\fm$ so that there is a short exact sequence of $R$-modules
\begin{equation} \tag{\ref{olsun}.1}
0 \to M_i \to M_{i-1} \to k^{\oplus r_i} \to 0
\end{equation}
for some integer $r_i\geq 0$. If $M_i \le \fn M_{i-1}$ for all $i = 1, \dotsc, t$, then $M_{i-1}/M_i \neq 0$ and hence $r_i\geq 1$. 
\end{chunk}

Next we extend \cite[2.5]{CT21} and observe that the filtration modules enjoy further properties if the module considered in Theorem~\ref{mainthm} is finitely generated and has positive depth:

\begin{cor} \label{rmk1} Let $f: (R, \fm) \to (S, \fn)$ be a local ring homomorphism of local rings and let $M$ be a finitely generated $S$-module. Assume the following hold:
\begin{enumerate}[\rm(i)]
\item $M$ is finitely generated over $R$ via $f$.
\item $\depth_R(M)\geq 1$.
\item There is a filtration of $M$ of the form
\[
M_t \le M_{t-1} \le \dotsb \le M_1 \le M_0 = M,
\] 
where each $M_i$ is a nonzero $S$-module and $\fm M_{i-1} \le M_i \le \fn M_{i-1}$ for all $i = 1, \dotsc, t$. 
\end{enumerate}
Then the following hold:
\begin{enumerate}[\rm(a)]
\item If $M_{i-1}$ is strongly-rigid over $R$ for some $i$, where $1\leq i \leq t$, then $M_i$ is strongly-rigid over $R$.
\item If $M_{i-1}$ is Tor-rigid over $R$ for some $i$, where $1\leq i \leq t$, and $\Tor_n^R(M_i, N)=0$ for some $n\geq 0$ and for some finitely generated $R$-module $N$, then $\pd_R(N)\leq n$ so that $\Tor_j^R(M_i, N)=0$ for all $j\geq n$.
\item If $M_{i-1}$ is Tor-rigid over $R$ for some $i$, where $1\leq i \leq t$, then $M_i$ is strongly-rigid and Tor-rigid over $R$.
\item It follows that:
$$ M \text{ is Tor-rigid over $R$} \Longrightarrow M_1 \text{ is rigid-test over $R$}  \Longrightarrow \cdots \Longrightarrow M_t \text{ is rigid-test over $R$}.$$
\end{enumerate}
\end{cor}

\begin{proof} Part (c) is an immediate consequence of part (b), and part (d) is a consequence of parts (a) and (c). Hence we proceed to prove parts (a) and (b).

Assume $\Tor_n^R(M_i, N)=0$ for some $n\geq 0$ and for some finitely generated $R$-module $N$. Then, since $M$ has positive depth, Theorem~\ref{Mainthm2}(a) implies that $\Tor_n^R(M_{i-1}, N)=0$. 

If $M_{i-1}$ is strongly-rigid, then $\pd_R(N)<\infty$ since $\Tor_n^R(M_{i-1}, N)=0$, and this proves part (a). 

Next assume $M_{i-1}$ is Tor-rigid. Then, as we have $\Tor_n^R(M_{i-1}, N)=0$, it follows that $\Tor_j^R(M_{i-1}, N)=0$ for all $j\geq n$. Hence (\ref{olsun}.1) yields the exact sequence $$0=\Tor^R_{n+1}(M_{i-1},N) \to \Tor_{n+1}^R(k^{\oplus r_i}, N) \to \Tor_{n}^R(M_i, N)=0$$ which implies $\pd_R(N)\leq n$ and the vanishing of $\Tor_j^R(M_i, N)$ for all $j\geq n$. This proves part (b).
\end{proof}


\begin{rmk} Under the setup of Corollary~\ref{rmk1}, if $M$ is Tor-rigid and $\Tor_n^R(M_i, M_j)=0$ for some $i\geq 1$, $j\geq 1$, and $n\geq 1$, then $M_i$ and $M_j$ have finite projective dimension and this forces $R$ to be regular. 
\end{rmk}

\begin{cor} \label{cor5} Let $f: (R, \fm) \to (S, \fn)$ be a local ring homomorphism of local rings and let $M$ be a finitely generated $S$-module. Assume the following hold:
\begin{enumerate}[\rm(i)]
\item $M$ is finitely generated over $R$ via $f$.
\item $M$ is Tor-rigid as an $R$-module and $\depth_R(M)\geq 1$.
\item There is a filtration of $M$ of the form
\[
M_t \le M_{t-1} \le \dotsb \le M_1 \le M_0 = M,
\] 
where each $M_i$ is a nonzero $S$-module and $\fm M_{i-1} \le M_i \le \fn M_{i-1}$ for all $i = 1, \dotsc, t$. 
\end{enumerate}
If $\Ext^n_R(N, M_i)=0$ or $\Tor_n^R(N, M_i)=0$ for some finitely generated $R$-module $N$, $i\geq 1$, and $n\geq 1$, then $\pd_R(N)\leq n-1$. So $R$ is regular if $\Ext^n_R(M_i, M_j)=0$ or $\Tor_n^R(M_i, M_j)=0$ for some $i\geq 1$, $j\geq 1$, and $n\geq 1$.
\end{cor}

\begin{proof} Let us first note that $\depth_R(M_i)=1$ for all $i=1, \ldots, t$; see the short exact sequence (\ref{olsun}.1). We know, by Corollary~\ref{rmk1}(d), that each $M_i$, with $i\geq 1$, is a rigid-test module over $R$. If $\Ext^n_R(N, M_i)=0$, then it follows from \cite[6.1]{HD2} that $\pd_R(N)\leq n-1$. Moreover, if $\Tor_n^R(N, M_i)=0$, then we conclude by Corollary~\ref{rmk1}(b) that $\pd_R(N)\leq n$. For the last claim about the regularity of $R$, see Corollary~\ref{cor3}(1).
\end{proof}

\begin{cor} \label{cor6} Let $f: (R, \fm) \to (S, \fn)$ be the Frobenius map $F^e$, where $R=S$ has prime characteristic $p$ and $R$ is an F-finite local complete intersection ring of positive depth. Let $e \geq 1$ and $v \geq 0$ be integers. Assume there is a filtration of $M=\up{v}{S}$  of the form
\[
M_t \le M_{t-1} \le \dotsb \le M_1 \le M_0 = M,
\] 
where each $M_i$ is a nonzero $S$-module and $\fm M_{i-1} \le M_i \le \fn M_{i-1}$ for all $i = 1, \dotsc, t$.  \footnote{\label{ft} This implies that $\up{v}{\big(  \fm^{[p^{e+v}]} \big)}  \le M_1 \le \up{v}{\big(  \fm^{[p^{v}]} \big)}$.} If $\Ext^n_R(\up{e}{N}, M_j)=0$ or $\Tor_n^R(\up{e}{N}, M_j)=0$ for some nonzero finitely generated $S$-module $N$, $j\geq 1$, and $n\geq 1$, then $R$ is regular. 
\end{cor}
\begin{proof} It follows from Corollary~\ref{cor5} that $\pd_R(\up{e}{N})<\infty$. Hence $R$ is regular due to \cite[1.1]{AHIY}.
\end{proof}

In the following we give some specific examples of filtration modules that are test and strongly-rigid. 

\begin{cor} \label{cor1} Let $(R, \fm)$ be a Cohen-Macaulay local ring of prime charactistic $p$, and let $M\neq 0$ and $N\neq 0$ be finitely generated $R$-modules such that $\depth_R(M)\geq 1$. Let $e\geq 0$ and $t\geq 1$ be integers. 
\begin{enumerate}[\rm(i)]
\item If $\Tor^R_n\big(\up{e}{(\fm^tM)}, N\big)=0$, or $\Tor^R_n\bigg( \up{e}{ \Big(\big(\fm^{[p^e]}\big)^tM\Big)}, N\bigg)=0$ for some $n\geq 1 $, then $\Tor^R_n(\up{e}{M}, N)=0$.
\item If $R$ is F-finite and $\up{e}{M}$ is a test module (respectively, a strongly-rigid module), then $\up{e}{(\fm^tM)}$ and $\up{e}{ \Big(\big(\fm^{[p^e]}\big)^tM\Big)}$ are test modules (respectively, are strongly-rigid modules).
\end{enumerate}
\end{cor}

\begin{proof} The claim in part (i) is a direct consequence of Theorem~\ref{Mainthm2} because $\up{e}{(\fm^tM)}$ and $\up{e}{ \Big(\big(\fm^{[p^e]}\big)^tM\Big)}$ are specific examples of filtration modules $M_i$ of the module considered in the theorem. 
Part (ii), in view of Defintion \ref{def}, follows from part (i).
\end{proof}

\begin{rmk} \label{c3} Let $R$ be an F-finite ring of prime charactistic $p$, $M= T \oplus \Omega_R T$ for some finitely generated $R$-module $T\neq0$, and let $e\geq 1$. 

There is an exact sequence of $R$-modules $0 \to \up{e} (\Omega_R T) \to \up{e}F \to \up{e}T \to 0$, where $F$ is a free $R$-module. Thus, if $\Tor_n^R(\up{e}M, N)=0$ for some $n\geq 1$, then $\Tor_n^R(\up{e} (\Omega_R T), N)=\Tor_n^R(\up{e}T, N)=0$ so that $\Tor_n^R(\up{e}F, N)=0$, which implies that $\Tor_n^R(\up{e}R, N)=0$. 
\end{rmk}

We also need the following results of Koh-Lee \cite{KohLe} and Avramov-Miller \cite{AvMi} in the sequel:

\begin{chunk} \label{c1} Let $(R, \fm)$ be an F-finite Cohen-Macaulay local ring of prime charactistic. 
\begin{enumerate}[\rm(i)]
\item (\cite[2.6]{KohLe}) If $\dim(R)=1$, then $\up{e}R$  is strongly-rigid (and hence test) for each $e\gg 0$.
\item (\cite[Main Thm]{AvMi}) If $R$ is a complete intersection, then $\up{e}R$ is a rigid-test module for each $e\geq 1$. 
\end{enumerate}
\end{chunk}

\begin{cor} \label{cor2} Let $(R, \fm)$ be an F-finite Cohen-Macaulay local ring of prime charactistic $p$, and let $N$ and $T$ be finitely generated $R$-modules such that $T\neq 0$ and $\depth_R(T)\geq 1$. Let $M= T \oplus \Omega_R T$, $e\geq 1$, and $t\geq 1$.
\begin{enumerate}[\rm(i)]
\item If $R$ is a complete intersection, then $\up{e}M$ is a rigid-test module. 
\item Assume $\dim(R)=1$, $e\gg 0$, $X$ is $\up{e}M$, $\up{e}{(\fm^tM)}$ or $\up{e}{ \Big(\big(\fm^{[p^e]}\big)^tM\Big)}$. If $\Tor_n^R(X, N)=0$ for some $n\geq 1$, then $\pd_R(N)\leq 1$.
\item Assume $R$ is a complete intersection, $X$ is $\up{e}{(\fm^tM)}$ or $\up{e}{ \Big(\big(\fm^{[p^e]}\big)^tM\Big)}$. If $\Tor_n^R(X, N)=0$ for some $n\geq 1$, then $\pd_R(N)\leq n$.
\end{enumerate}
\end{cor}

\begin{proof} (i) The claim that $\up{e}M$ is a rigid-test module follows from Remark~\ref{c3} and \ref{c1}(ii). 

(ii) It follows from \ref{c1}(i) that $\up{e}R$ is strongly-rigid. Hence Remark \ref{c3} shows that $\up{e}M$ is also strongly-rigid. Now we conclude from Corollary \ref{cor1}(ii) shows that $\up{e}{(\fm^tM)}$ and $\up{e}{ \Big(\big(\fm^{[p^e]}\big)^tM\Big)}$ are both strongly-rigid. Therefore, if $\Tor_n^R(X, N)=0$ for some $n\geq 1$, then $\pd_R(N)<\infty$ so that $\pd_R(N)\leq 1$ since $\depth(R)\leq 1$.

(iii) We know by part (i) that $\up{e}M$ is a rigid-test module. In particular $\up{e}M$ is Tor-rigid. Note also that $\depth_R(\up{e}M)\geq 1$. So, if $\Tor_n^R(X, N)=0$ for some $n\geq 1$, we see from Corollary \ref{rmk1}(b) that $\pd_R(N)\leq n$.
\end{proof}


\section{On torsion in tensor products of modules}

In this section we consider the following conjecture of Huneke and Wiegand:

\begin{conj} [{\cite[page 473-474]{HW1}}] \label{HWC} Let $R$ be a one-dimensional local ring and let $M$ be a finitely generated $R$-module which has rank. If $M\otimes_RM^{\ast}$ is nonzero and torsion-free, where $M^{\ast}=\Hom_R(M,R)$, then $M$ is free. 
\end{conj}

Recall that a module $M$ over a ring $R$ is said to have \emph{rank} if there is a nonnegative integer $r$ such that $M_{\fp} \cong R_{\fp}^{\oplus r}$ for each associated prime ideal $\fp$ of $R$. For example, if $M$ has finite projective dimension, or $R$ is a domain, then $M$ has rank. Conjecture~\ref{HWC} fails if the module considered does not have rank, or the ring in question has dimension at least two; see, for example, \cite[8.5 and page 447]{CeRo}. Note that, in Conjecture~\ref{HWC}, it suffices to additionally assume $M$ is a torsion-free module which has positive constant rank. 

Conjecture~\ref{HWC} is wide open in general. It turns out that, over Gorenstein rings, the conjecture is a special case of a celebrated conjecture of Auslander and Reiten \cite{AuRe} on the vanishing of cohomology that stems from the representation theory of finite dimensional algebras; this is one of the main motivations to study Conjecture~\ref{HWC}; see also \cite[8.6]{CeRo}.

Huneke and Wiegand \cite[3.1]{HW1} proved that Conjecture~\ref{HWC} holds over hypersurface rings. There are also several other cases where the conjecture holds, for example, if $R$ is a Cohen-Macaulay local ring of minimal multiplicity \cite[3.6]{HSW}, or $M$ is an  integrally closed $\fm$-primary ideal \cite[2.17]{CGTT}. Note that Conjecture~\ref{HWC} holds if the module $M$ in question is strongly-rigid \cite[2.15]{CGTT}. Therefore, Corollary~\ref{cor2} yields new classes of modules establishing the conjecture over local rings of prime characteristic.   

If $M$ is a finitely generated module over a local ring $(R, \fm)$ is a local ring such that $M$ has rank and $0\neq M \otimes_R M^{\ast}$ is torsion-free, then $\Supp_R(M)=\Spec(R)$; see, for example, \cite[1.3]{GO}. Therefore the following result of Dey and Kobayashi \cite{DeyTos} establishes Conjecture~\ref{HWC} for nonzero modules that are of the form $\fm N$.

\begin{thm} [{\cite[1.5(1)]{DeyTos}}] \label{DT} Let $(R, \fm)$ be a local ring of depth one and let $M =\fm N$ for some finitely generated $R$-module $N$. Assume $\Supp_R(M)=\Spec(R)$ and $M_{\fp}$ is free over $R_{\fp}$ for each associated prime $\fp$ of $R$. If $M \otimes_R M^{\ast}$ is nonzero and torsion-free, then $M$ is free and $R$ is regular. 
\end{thm}

\begin{lem} \label{LT} Let $(R, \fm)$ be a local ring of positive depth and let $M$ be a nonzero finitely generated $R$-module. Assume there is a filtration of $M$ of the form
\[
M_t \le M_{t-1} \le \dotsb \le M_1 \le M_0 = M,
\] 
where each $M_i$ is a nonzero $R$-module and $\fm M_{i-1} \le M_i$ for all $i = 1, \dotsc, t$. Then it follows
$$ M \text{ satisfies } (\PP) \Longrightarrow M_1 \text{ satisfies } (\PP)   \Longrightarrow \cdots \Longrightarrow M_t \text{ satisfies } (\PP),$$
where the property $(\PP)$ denotes one of the following: (a) being not torsion, (b) being faithful, (c) being locally free on the set of associated primes of $R$, or (d) having rank.
\end{lem}

\begin{proof} Assume $M$ satisfies $(\PP)$. Note that it is enough to show $M_1$ satisfies $(\PP)$.

For part (a), if $M_1$ is torsion, then $x M_1=0$ for some non zero-divisor $x \in \fm$. Since $xM\subseteq \fm M \subseteq M_1$, we see that $x^2M=0$ and hence $M$ is torsion. This shows that $M_1$ is not torsion.

For part (b), since $\fm M \subseteq M_1$, it is enough to show that $\fm M$ is faithful. Let $y \in \Ann_R(\fm M)$. Then $y \fm M=0$ and hence $y \fm =0$ since $M$ is faithful. This implies $y=0$ as $\fm$ contains a non zero-divisor.

For parts (c) and (d), we see from the exact sequence (\ref{olsun}.1) that $M_{\fp} \cong (M_1)_{\fp}$ for each associated prime $\fp$ of $R$. This gives the required conclusions.
\end{proof}

We recall the following basic facts for the proof of Theorem \ref{TT}; note that $(-)^{\ast}=\Hom_R(-,R)$.

\begin{chunk}  \label{HWlem}  Let $R$ be a ring and let $M$ and $N$ be finitely generated $R$-modules.
\begin{enumerate}[\rm(i)]
\item Let $\tp R(M)$ denote the torsion submodule of $M$ and set $\overline{M}=M/ \tp R(M)$. Assume $0\neq N$ and $M\otimes_RN$ is torsion-free. Then $M\otimes_RN \cong \overline{M}\otimes_RN$, and $M$ is free if and only if $\overline{M}$ is free; see \cite[1.1]{HW1}.
\item If $N$ is \emph{torsionless}, that is, if the natural map $N \to N^{\ast\ast}$ is injective, then $N$ embeds into a free $R$-module so that it is torsion-free; see, for example, \cite[1.4.19]{BH}. Also if $N$ is torsion-free and locally free on each associated prime of $R$, then the kernel of $N \to N^{\ast\ast}$ is zero when localized at each associated prime of $R$, hence the kernel is torsion so that it vanishes and $N$ is torsionless.
\end{enumerate}
\end{chunk}

\begin{thm} \label{TT} Let $f: R \to S$ be a ring homomorphism, where $(R, \fm)$ is a local ring of positive depth, $S$ is Noetherian,  and $\fm S \subseteq \Jac(S)$. Let $M$ be a generated $S$-module. Assume:
\begin{enumerate}[\rm(i)]
\item There is a filtration of $M$ of the form
\[
M_t \le M_{t-1} \le \dotsb \le M_1 \le M_0 = M,
\] 
where each $M_i$ is a nonzero $S$-module and $\fm M_{i-1} \le M_i \le \Jac(S) M_{i-1}$ for all $i = 1, \dotsc, t$. 
\item $M$ is finitely generated as an $R$-module via $f$ and  $\Supp_R(M)=\Spec(R)$.
\item $M_{\fp}$ is free over $R_{\fp}$ for each associated prime $\fp$ of $R$.
\end{enumerate}
Then the following hold:
\begin{enumerate}[\rm(a)]
\item If $0\neq N$ is a finitely generated torsionless $R$-module and $M_i \otimes_RN$ is a torsion-free $R$-module, where $1\leq i \leq t$, then $N$ is free and $M_i$ is torsion-free.
\item If $M_i \otimes_RM_{j}$ is a torsion-free $R$-module, where $1\leq i, j \leq t$, then both $M_i$ and $M_j$ are free and $R$ is regular.
\item Assume $\depth(R)=1$. If $0\neq M_i \otimes_RM^{\ast}_{i}$ is a torsion-free $R$-module, where $1\leq i\leq t$, then $M_i$ is free and $R$ is regular. 
\end{enumerate}
\end{thm}

\begin{proof} (a) There is an exact sequence of finitely generated $R$-modules $0 \to N \to F \to C \to 0$ where $F$ is free; see~\ref{HWlem}(ii). Tensoring this sequence with $M_i$ over $R$, we obtain an injection $\Tor_1^R(M_i, C) \hookrightarrow M_i \otimes_RN$. By Lemma~\ref{LT}(c), for each associated prime $\fp$ of $R$, we have that $(M_i)_{\fp}$ is free over $R_{\fp}$ and so $\Tor_1^R(M_i, C)_{\fp}=0$. This implies that $\Tor_1^R(M_i, C)$ is a torsion $R$-module and hence it vanishes as $M_i \otimes_RN$ is torsion-free over $R$. Also $M$ is a faithful $R$-module; see, for example, \cite[2.12]{DeyTos}. So Lemma~\ref{LT}(b) shows that $M_{i-1}$ is faithful. As $\Tor_1^R(M_i, C)=0$, Theorem~\ref{Mainthm2}(b) implies that $\pd_R(C)\leq 1$. Now the sequence $0 \to N \to F \to C \to 0$ shows that $N$ is free. Consequently $M_i$ is torsion-free because $0\neq M_i \otimes_RN$ is torsion-free.

(b) Note that $(M_j)_{\fp} \cong (\overline{M_j})_{\fp}$ for each associated prime $\fp$ of $R$. Therefore Lemma~\ref{LT}(c) implies that $(\overline{M_j})_{\fp}$ is free over $R_{\fp}$ for each associated prime $\fp$ of $R$.  Therefore $\overline{M_j}$ is a torsionless $R$-module; see~\ref{HWlem}(ii). It follows that $0\neq M_i \otimes_R M_j \cong M_i\otimes_R \overline{M_j}$; see~\ref{HWlem}(i). Thus $\overline{M_j}\neq 0$, so part (a) implies that $\overline{M_j}$ is free. As $M_i\neq 0$, we conclude from~\ref{HWlem}(i) that $M_j$ is free. Hence $R$ is regular by Corollary~\ref{cor3}(1). By interchanging the roles of $M_i$ and $M_j$, we see that $M_i$ is free too. 

(c) As $0\neq M_i^{\ast}$ is a torsionless $R$-module, it follows from part (a) that $M_i^{\ast}$ is free over $R$ and $M_i$ is torsion-free over $R$. Then $M_i$ is free given that $\depth(R)=1$; see \cite[2.13]{CGTT}. So $R$ is regular by Corollary~\ref{cor3}(1). 
\end{proof}

We finish this section with the following corollary of Theorem \ref{TT}, which establishes Conjecture \ref{HWC} for certain filtration modules. 

\begin{cor} \label{last} Let $f: (R, \fm) \to (S, \fn)$ be a finite local ring homomorphism of local rings, where $R$ is a one-dimensional domain, and let $0\neq M$ be a  finitely generated $S$-module that is torsion-free as an $R$-module. Assume there is a filtration of $M$ of the form
\[
M_t \le M_{t-1} \le \dotsb \le M_1 \le M_0 = M,
\] 
where each $M_i$ is an $S$-module and $\fm M_{i-1} \le M_i \le \fn M_{i-1}$ for all $i = 1, \dotsc, t$. 
If $M_i \otimes_RM^{\ast}_{i}$ is a torsion-free $R$-module, where $1\leq i\leq t$, then $M_i$ is free.\end{cor}

\begin{proof} Note that $\fm^i M\subseteq M_i$. So $M_i\neq 0$ since $M$ is torsion-free over $R$. Hence $M_i$ is torsion-free over $R$. Then $M_i^{\ast}\neq 0$, otherwise $M_i$ would be torsion over $R$ which implies $M_i=0$. Thus $M_i \otimes_RM^{\ast}_{i}\neq 0$. On the other hand, as $M$ is a torsion-free $R$-module and $R$ is a domain, we have that $\Supp_R(M)=\Spec(R)$. Therefore Theorem \ref{TT}(c) implies that $R$ is regular and $M_i$ is free (and $M$ is also free as well). 
\end{proof}


\end{document}